\documentclass[twocolumn]{autart}    

\usepackage{amsmath}
\usepackage{amssymb}
\usepackage{amsfonts}
\usepackage{mathrsfs}
\usepackage{array}
\usepackage{subfigure}
\usepackage{graphicx}
\usepackage{extarrows,epstopdf}
\usepackage{natbib}

\setcitestyle{round,aysep={},yysep={;}}

\theoremstyle{definition}
\newtheorem{theorem}{Theorem}

\newtheorem{ass}{Assumption}

\newtheorem{lemma}{Lemma}
\newtheorem{definition}{Definition}
\newtheorem{remark}{Remark}

\newtheorem{proof}{Proof}
\UseRawInputEncoding

\newcommand{\beq}[1]{\begin{equation} \label{#1}}
\newcommand{\eeq}{\end{equation}}

\def\one{{\hbox{1{\kern -0.35em}1}}}

\allowdisplaybreaks[4]

\usepackage[colorlinks,
            linkcolor=blue,
            anchorcolor=blue,
            citecolor=blue]{hyperref}

\begin{document}

\begin{frontmatter}

\title{Identification of High-Dimensional ARMA Models with Binary-Valued Observations  \thanksref{footnoteinfo}} 

\thanks[footnoteinfo]{
Corresponding author: Yanlong Zhao.}

\author[AMSS]{Xin Li},\ead{lixin2020@amss.ac.cn}
\author[USTB1]{Ting Wang}\ead{wangting@ustb.edu.cn},    
\author[USTB3]{Jin Guo}, \ead{guojin@ustb.edu.cn}
\author[AMSS]{Yanlong Zhao}\ead{ylzhao@amss.ac.cn}               
\address[AMSS]{The Key Laboratory of Systems and Control, Institute of Systems Science,	Academy of Mathematics and Systems Science, Chinese Academy of Sciences, Beijing 100190, PR China}
\address[USTB1]{School of Intelligence Science and Technology, University of Science and Technology Beijing, Beijing 100083, PR China}
\address[USTB3]{School of Automation and Electrical Engineering, University of Science and Technology Beijing, Beijing 100083, PR China}


\begin{abstract}                          
This paper studies system identification of high-dimensional ARMA models with binary-valued observations.
The existing paper can only deal with the case where the regression term is only one-dimensional. In this paper, the ARMA model with arbitrary dimensions is considered, which is more challenging.
Different from the identification of FIR models with binary-valued observations, the prediction of original system output and the parameter both need to be estimated in ARMA models. An online identification algorithm consisting of parameter estimation and prediction of original system output is proposed. The parameter estimation and the prediction of original output are strongly coupled but mutually reinforcing.  By analyzing the two estimates at the same time instead of analyzing separately, we finally prove that  the parameter estimate can converge to the true parameter with convergence rate $O(1/k)$ under certain conditions. Simulations are given to demonstrate the theoretical results.
\end{abstract}

\end{frontmatter}

\section{Introduction}

With the development of informatization and digital communications, set-valued systems have increasingly emerged in our life due to its wide application (\cite{le2010system}). Examples of such systems include switching sensors for measuring automobile exhaust gas (\cite{le2003system},\cite{wang2002prediction}), the healthy or unhealthy state in complex diseases (\cite{bradfield2012genome,bi2015svsi}) and target recognition systems with multiple target categories (\cite{wang2016radar}).   Different from traditional systems with accurate output, the information that set-valued systems provide is which set the output belongs to. If output can only belong to two sets, the system is called the binary-valued observation system.  For example, an oxygen sensor in automative emission control can only obtain whether the oxygen content is bigger than a certain threshold. The identification of set-valued systems is more difficult since the obtained information is less than the traditional systems.

Recently, researches on identification of set-valued systems constitutes a vast body of literature.  \cite{le2003system} proposed the empirical measure method to investigate the identification errors, time complexity, input design,
and impact of disturbances and unmodeled  dynamics on identification accuracy and complexity for linear systems with binary-valued information. The asymptotic efficiency of this method require the inverse of the distributed function of the noise be uniformly bounded. \cite{wang2018asymptotically} proposed a non-truncated empirical measure method for finite impulse response (FIR) systems  and the asymptotically efficiency in sense of Cram\'{e}r-Rao (CR) lower bound was proved. \cite{casini2007time} studied worst-case identification of systems equipped with binary-valued sensors. An upper bound on time complexity was given for identification of FIR systems and the optimal input design problem was solved for gain systems.  \cite{godoy2011identification} introduced the expectation maximization (EM) algorithm into identification of quantized systems and they gave some simulation results to demonstrate the convergency. In \cite{zhao2016iterative}, an EM-type algorithm was proposed for FIR systems and it was proved that the algorithm is convergent with exponential convergence rate. \cite{you2015recursive} proposed a recursive estimator using the set-valued innovations under the minimum mean square error criterion. The quantization scheme was designed by the prediction of system output, which makes quantization thresholds variable. For quantizers with fixed thresholds, \cite{guo2013recursive} proposed a recursive projection algorithm for FIR models, which was proved to be convergent with convergence rate $O(\ln k/k)$. \cite{wang2019adaptive} extended the algorithm to the case of matrix input and vector output, and  gave a faster convergence rate $O(1/k)$.

\nocite{zhang2019fir} 
\nocite{zhang2019asymptotically}

However, most of the literature we mentioned before consider FIR  systems with set-valued observations. There are only few works studying quantized identification of  autoregressive moving average (ARMA) systems. Those are \cite{marelli2013identification} and  \cite{yu2016quantized}).
\cite{marelli2013identification} quantified the effect of finite-level quantization and packet dropouts on identification of ARMA models. Using the maximum likelihood criterion like \cite{godoy2011identification}, an algorithm switching from the EM based method to the quasi-Newton-based method is proposed. When the number of sample $N$ is smaller than a certain value $N_0$, the EM-based method is used to estimate the parameter. When $N$ is bigger than $N_0$, the quasi-Newton-based method is used with the initial value being the EM-based estimate. However, how to determine the value of $N_0$ is not given in the paper.  Besides, the quasi-Newton-based method cannot ensure the convergence of the estimate error even if the initial value is the true parameter.
\cite{yu2016quantized} studied identification of ARMA models with colored measurement noise and time-varying quantizers. The estimator and the quantizer are jointly designed to identify the system. However, there exist time-invariant quantizers with constant thresholds in practical systems, which cannot be designed. If the system input keeps unchanged, the time-varying quantizers can provide certain information by changing the thresholds, but the time-invariant quantizers can provide very limited information since the thresholds are constant. The identification with time-invariant quantizers in our paper will be more difficult since the information is less.  \cite{yu2016quantized} mentioned the time-invariant quantizer is easy to implement but at the expense of infinite quantization levels, which implies the time-invariant quantizers with finite quantization levels may not achieve identification.

 This paper studies the identification of high-dimensional ARMA models with time-invariant and binary-valued quantizers. The quantizers are easy to implement, but the accessible information is much less than that with infinite-level quantizers, which brings difficulties for identification. Different from FIR systems in \cite{guo2013recursive}, the identification of ARMA models is more complicated  since the original system output and the parameter both need to be estimated by using binary-valued observations.

The parameters of high-dimensional ARMA models are identified by a recursive projection algorithm which only needs binary-valued observations. 
The main contributions of this paper can be summarized as follows:
\begin{enumerate}
  \item[i)] Compared to \cite{wang2023identification}, the identification of high-dimensional ARMA models with binary-valued observations under persistent excitation and fixed threshold is studied for the first time. 
  Since the Assumption 1 in \cite{wang2023identification} cannot be true for high dimensions definitely, in this paper, the ARMA model with arbitrary dimensions is considered, which is more challenging.
Different form  FIR systems, the correlation of the vector with system input and unknown system output need to be considered. Meanwhile, the binary observations with a fixed threshold supply less information than the ones with adjustable thresholds.   
 
  \item[ii)] An online identification algorithm is proposed for the high-dimensional ARMA models by using only binary-valued observations. The algorithm consists of parameter estimation and prediction of original system output. Using the prediction of original system output, the system input and the binary-valued observation, the parameter is estimated by the recursive projection algorithm. 
Based on the parameter estimate and the system input, the original system output is predicted according to the system model.

  \item[iii)]  The proposed algorithm is proved to be convergent under certain conditions and the convergence rate can achieve $O(1/k)$ which is at the same order as that for FIR systems in \cite{wang2019adaptive}. More than this, when the coefficients of the regression term are small, the model we consider degenerates into FIR model. In this case, our conclusion is consistent with that in \cite{wang2019adaptive}. The main difficulty of analysis lies in the strong coupling of parameter estimation and prediction of original system output. Due to the same difficult in \cite{wang2023identification} that the Euclidean-norm of matrix $A$ is more than 1, making it impossible to achieve convergence analysis using one-step iteration. In this paper, a more accurate upper bound for the output estimation error is given and based on this, convergence analysis of parameter estimation is provided.
\end{enumerate}
The rest of this paper is organized as follows. Section 2 formulates the identification of high-dimensional ARMA models with binary-valued observations. Section 3 gives an online identification algorithm based on prediction of the original system output. Section 4 introduces the properties of the identification algorithm, including convergence and convergence rate. In Section 5, simulations are given to demonstrate the theoretical results.  Section 6 concludes this paper and gives some related work.

\section{Problem formulation}

Consider the high-dimensional ARMA model with binary-valued observations:
\begin{equation}\label{sys:ARMA1}
	\begin{cases}
	\begin{aligned}
		y_{k}&=A(z)y_{k-1}+B(z)u_{k},\\
		z_{k}&=y_{k}+d_{k},\\
		s_{k}&=I_{\left\lbrace z_{k}\leq C \right\rbrace },
	\end{aligned}
	\end{cases}
\end{equation}
where $z$ is a delay operator, $ A(z)=a_{1}+a_{2}z+\cdots+a_{p}z^{p-1} $, $a_p\neq 0$, $ B(z)=b_{1}+b_{2}z+\cdots+b_{q}z^{q-1} $, $ 1-zA(z) $ and $B(z) $ have no common roots, $d_{k}\in \mathbb{R}$ is the noise, $u_{k}\in \mathbb{R} $ is the input, $s_{k}\in \mathbb{R}$ is the binary-valued observation, $C\in \mathbb{R}$ is the fixed threshold, and $I_{\{z_k\leq C\}}$ is defined as
$$I_{\{z_k\leq C\}}=\begin{cases}
1, & \quad z_k \leq  C; \\
0, & \quad z_k>C.\\
\end{cases}$$
Let $ \phi_{k}=\left[ y_{k-1}, \ldots, y_{k-p}, u_{k}, \ldots, u_{k-q+1} \right]^{T} $,
$ \theta=[a_{1}, \ldots ,$ $a_{p}, b_{1}, \ldots, b_{q} ] ^{T} $. Then, system \eqref{sys:ARMA1} can be written as 
\begin{equation}\label{sys:theta}
	\begin{cases}
	\begin{aligned}
		z_{k}&=\phi_{k}^{T}\theta+d_{k},\\
		s_{k}&=I_{\left\lbrace z_{k}\leq C \right\rbrace }.
	\end{aligned}
	\end{cases}
\end{equation}
Our goal is to estimate the parameter $ \theta $ by using the binary-valued observation $s_k$ and input $u_k$.

We have the following assumptions on the parameters, noise and input.

\begin{ass}\label{ass:A}
The high-dimensional ARMA model is causal, i.e. all roots of \eqref{ARmodel} are outside of the unit circle.
\begin{equation}\label{ARmodel}
1-zA(z)=1-a_{1}z-\cdots-a_{p-1}z^{p-1}-a_{p}z^{p}=0.
\end{equation}
\noindent Further, assume that all roots of \eqref{ARmodel} must be greater than a given positive number $1/h$, where $0<h<1$.
\end{ass}

\begin{remark}[\cite{woodward2017applied}, Theorem 3.2] \label{rem:yk}
The causality of ARMA model in Assumption \ref{ass:A} is a commonly used property, which ensures the stability of the output $y_k$ {\rm \citep{Brockwell2016ARMA}}.
Under Assumption \ref{ass:A}, the original system output $y_k$ in system \eqref{sys:ARMA1} are bounded if 
$u_k$ are bounded.
\end{remark}

\begin{remark}\label{rem:stable}
Let $ A=\begin{bmatrix}
		a_{1} & a_{2} & \cdots & a_{p}\\
		1 & 0 & \cdots & 0\\
		&\ddots & \ddots & \vdots\\
		&& 1 & 0
	\end{bmatrix} $ and denote $ \left\lbrace \lambda_{i},i=1,\ldots,p \right\rbrace  $ as the eigenvalues of A. Since
\begin{equation}
\begin{aligned}
\vert \lambda I-A \vert=
\begin{vmatrix}
\lambda-a_{1} & -a_{2} & \cdots & -a_{p}\\
-1 & \lambda & \cdots & 0\\
&\ddots & \ddots & \vdots\\
&& -1 & \lambda
\end{vmatrix}\\
=\lambda^{p}-a_{1}\lambda^{p-1}-\cdots-a_{p-1}\lambda-a_{p},
\end{aligned}
\end{equation}
it follows that $  \lambda_{1},\ldots,\lambda_{p}  $ are the reciprocal of the roots of the polynomial $ 1-zA(z) $.
Hence, under Assumption \ref{ass:A}, we know that 
$ |\lambda_{i}|<h<1 $ for $ i=1,\ldots,p, $ and the spectral radius of matrix A, denoted as $ \rho(A) $, satisfies
$$
\rho(A)=\mathop{\max}\limits_{i=1,\ldots,p} |\lambda_{i}|<h<1.
$$

\end{remark}

\begin{remark}\label{rem:a2leqh}
Under Assumption \ref{ass:A}, Remark \ref{rem:stable} and Vieta-Theorem for high-order polynomials, we have
$$
\begin{aligned}
|a_1|&=|\sum_{i=1,\ldots,p}\lambda_i|\leq ph,\\
|a_2|&=|\sum_{i,j=1,\ldots,p,i<j}\lambda_i \lambda_j|\leq \dfrac{p(p-1)h^2}{2},\\
&\vdots\\
|a_p|&=|\prod_{i=1,\ldots,p}\lambda_i|\leq h^p.
\end{aligned}
$$
and let
$$
\sum_{i=1,\ldots,p}a_{i}^{2}\leq p^2 h^2+\cdots+h^{2p}=g(h),
$$
where $ g(h) $ is a definite and known upper bound of $\sum_{i=1,\ldots,p}a_{i}^{2}$.
\end{remark}
\begin{lemma} [\cite{Lei2020}, Lemma 2.4.1]\label{fanshu}
For any $\epsilon>0$ and $A\in \mathbb{R}^{p\times p}$,  there exist $M=\sqrt{p}\left(1+\dfrac{2}{\epsilon}\right)^{p-1}$ and $h_1=\rho(A)+\epsilon \|A\|$, such that
$$ \|A^k\|\leq M h_1^k, \forall k\geq 0,
$$ where  $\|\cdot\|$ is Euclidean-norm in this paper. 
\end{lemma}

\begin{remark}\label{rem:Akfanshu}
Since 
 $$
 \begin{aligned}
A^T A 
	=&\begin{bmatrix}
		a_{1} \\ a_{2} \\ \vdots \\ a_{p}\\
	\end{bmatrix}
	\begin{bmatrix}
		a_{1} & a_{2} & \cdots & a_{p}\\
	\end{bmatrix}
	+\begin{bmatrix}
		1 &  &  & \\
			&\ddots &  & \\
			 &  & 1 & \\
		&&  & 0
	\end{bmatrix},
\end{aligned}
   $$
   and
   $$
   \begin{aligned}
     &\lambda_{max}\left(A^T  A \right)\\
     \leq& \lambda_{max}\left(\begin{bmatrix}
		a_{1} \\ a_{2} \\ \vdots \\ a_{p}\\
	\end{bmatrix}
	\begin{bmatrix}
		a_{1} \\ a_{2} \\ \vdots \\ a_{p}\\
	\end{bmatrix}^T
	 \right)+\lambda_{max}\left(\begin{bmatrix}
		1 &  &  & \\
			&\ddots &  & \\
			 &  & 1 & \\
		&&  & 0
	\end{bmatrix} \right) \\
	=&1+\sum_{i=1,\ldots,p}a_{i}^{2},
   \end{aligned}
   $$
hence, by Remark \ref{rem:a2leqh}, we have
$$
\|A\|=\sqrt{\lambda_{max}\!\left(A^T \! A \right)}\leq \sqrt{1+g(h)}.
$$
In Lemma \ref{fanshu}, we can choose $ \epsilon $ such that $ h_1<1 $. Thus, $ \|A^k\|\leq M h_1^k$, and $M, h_1$ are known.
\end{remark}

\begin{ass}\label{ass:Theta}
Let $\Omega$ denote the set of parameters satisfying Assumption \ref{ass:A}.
The parameter $\theta$ belongs to a convex compact set $\Theta $ and $\Theta$ is the subset of $\Omega$, i.e.
$\theta\in \Theta\subseteq\Omega.$
What's more,
$
B=\mathop{sup}\limits_{v \in \Theta}\parallel v \parallel.
$

\end{ass}

\begin{ass}\label{ass:noise}
	The noise $ \left\lbrace d(k), k = 1, 2,\ldots \right\rbrace $ is a sequence of independent and identically distributed(i.i.d) random variables with $ E\{d_{k}\}=0 $ and known distribution function $ F(\cdot) $ , and the associated density function is continuous with $ f(x) = {\rm d} F(x)/{\rm d} x \neq 0 $.
\end{ass}

\begin{remark}
  In Assumption \ref{ass:noise}, we assume the distribution function is known, which is an important factor in our approach to derive estimates of the unknown parameter $\theta$. If the noise $d$ is normally distributed and the variance is unknown, we can use the method in \cite{le2006noise} to transform the system model into one with known noise, and then estimate the original variance and unknown parameters simultaneously.
\end{remark}

\begin{ass}\label{ass:uk}
The input $ \left\lbrace u(k), k = 1, 2,\ldots \right\rbrace $ is bounded with $ | u_{k} | \leq M < \infty $. And there exists a constant $\delta_1>0$ and an integer $m>0$ such that
$$ \dfrac{1}{m}\sum_{i=k}^{k+m-1}U_iU_i^T\geq \delta_1I_{p+q}, \quad k=1, 2, \ldots,$$
where $U_i=[u_i, u_{i-1}, \ldots, u_{i-p-q+1}]^T$, and $I_{p+q}$ is the  $(p+q)$-dimensional identity matrix.
\end{ass}

\section{Identification algorithm}

In this section, the parameter estimator and the predictor of original output will be jointly designed. First, the definition of the projection operator is given.
\begin{definition}\label{def:operator}
	For a given convex compact set $\Lambda\subset \mathbb{R}^n$, the projection operator $\Pi_{\Lambda}$ is defined as
	$\Pi_{\Lambda}(x)=arg\min_{\lambda\in\Lambda}\|x-\lambda\|$,  $\forall x\in \mathbb{R}^n$.
\end{definition}
The projection operator has the following property.
\begin{remark}[\cite{calamai1987projected}]
	The projection operator given in Definition \ref{def:operator} follows
	$\|\Pi_{\Lambda}(x_1)-\Pi_{\Lambda}(x_2)\|\leq \|x_1-x_2\|$, $\forall x_1, x_2\in \mathbb{R}^n$.
\end{remark}
Then, the identification algorithm including parameter estimator $ \hat{\theta}_{k}=[\hat{a}_{1}^{k},\ldots,\hat{a}_{p}^{k},\hat{b}_{1}^{k},\ldots,\hat{b}_{p}^{k}]^{T} $ and output predictor $ \hat{y}_{k-1} $ is given as follows.
\begin{equation}\label{algorithm}
	\left\{
	\begin{aligned}
		&\hat{\theta}_{k}=\Pi _{\Theta}\left\lbrace \hat{\theta}_{k-1}+\dfrac{\gamma}{k} \varphi_{k}(F(C-\varphi_{k}^{T}\hat{\theta}_{k-1})-s_{k})\right\rbrace, \\
		&\varphi_{k}=\left[\hat{y}_{k-1},\ldots,\hat{y}_{k-p},u_{k},\ldots,u_{k-q+1} \right] ^{T},\\
		&\hat{y}_{k-1}=\varphi_{k-1}^{T}\hat{\theta}_{k-1},
	\end{aligned}
	\right.
\end{equation}
where $ \gamma $ is the step size that can be designed.

\begin{remark}\label{rem:hyk}
The prediction of the original output $\hat{y}_{k} $ is bounded since $\hat{\theta}_k\in \Theta$ satisfies stability condition (see Assumptions \ref{ass:A} and \ref{ass:Theta}).
\end{remark}

\begin{remark}\label{rem:M12}
By Remarks \ref{rem:yk}, \ref{rem:hyk}, and Assumption \ref{ass:uk}, there exist constant $M_1>0$ and $M_2>0$ such that
$\|\phi_{k}\|\leq M_1$, $\|\varphi_k\|\leq M_2$,
where $\phi_{k}=[y_{k-1}, \ldots, y_{k-p}, u_k, \ldots, u_{k-q+1}]^T$,
$\varphi_{k}=[\hat{y}_{k-1}, \ldots, \hat{y}_{k-p}, u_k, \ldots, u_{k-q+1}]^T$.
\end{remark}

\section{Properties of the identification algorithm}
In this section, the convergence and convergence rate of the identification algorithm will be obtained.
Denote $e_k=\hat{\theta}_k-\theta$ as the estimation error. We have

\begin{lemma}[\cite{guo2013recursive}, Lemma 8]\label{lem:theta}
	Under Assumptions \ref{ass:Theta}, \ref{ass:noise} and \ref{ass:uk}, we have
	\begin{center}
	$\|e_i-e_j\|\leq \dfrac{(i-j)M_2}{j+1},$
	for $i\geq j\geq 0.$
	\end{center}
\end{lemma}

\begin{lemma}[\cite{chen1987adaptive}, Lemma 1]\label{lem:delta}
Under Assumption \ref{ass:uk}, there exists a constant $c_0>0$  s.t.
\begin{align*}
\lambda_{\min}\left(\sum_{i=k}^{k+p+m-1}\phi_{i}\phi_{i}^T\right)&\geq c_0\lambda_{\min}\left(\dfrac{1}{m}\sum_{i=k+p}^{k+p+m-1}U_iU_i^T\right)\\
&\geq c_0\delta_1,\quad \forall k>0,
\end{align*}
which implies that there exist constant $N=p+m$ and $\delta=\dfrac{c_0\delta_1}{N}$ such that
$$\dfrac{1}{N}\sum_{i=k}^{k+N-1}\phi_i\phi_i^T\geq \delta I_{p+q},\quad \forall k>0,$$
where $\phi_i$ and $p$ are defined in system \eqref{sys:ARMA1},  $U_i$, $m$ and $\delta_1$ are defined in Assumption \ref{ass:uk}.
\end{lemma}

%

\begin{lemma}\label{lem:difference_rk}
For the iteration
\begin{equation}\label{rkequation}
r_{k}  =  \left(1 - \dfrac{\eta_{1}}{k}\right) r_{k-1} 
 +\dfrac{\eta_{2}}{k} \sum_{i=2}^{k}h_1^{i-1}r_{k-i} 
 + O\left( \dfrac{1}{k^{2}} \right),
\end{equation}
	where $ \eta_{1}$, $\eta_{2}>0$, $ 0<h_1<1 $, and $ r_0 >0$. We can get the following assertions:
	
	 $ r_{k} $ converges to zero if and only if  $ \eta_1 -\dfrac{\eta_{2}h_{1}}{1-h_{1}}>0 $, and we have
\begin{align*}
r_{k}=\begin{cases}
O\left(\dfrac{1}{k^{\eta}}\right), &0<\eta<1;\\
O\left(\dfrac{ln k}{k}\right), &\eta=1;\\
O\left(\dfrac{1}{k}\right), &\eta>1.
\end{cases}
\end{align*}
where $ \eta=\eta_1 -\dfrac{\eta_{2}h_{1}}{1-h_{1}}$.
\end{lemma}
The proof of Lemma \ref{lem:difference_rk} is put in Appendix \ref{app:A}.

\begin{lemma}\label{lem:xn}
	For a sequence $\{x_n>0, n=1, 2, \ldots \}$, we have:
	$$\dfrac{x_n+x_{n+1}+\cdots+x_{n+N-1}}{N}=O\left(\dfrac{1}{n^t}\right)\Leftrightarrow x_n=O\left(\dfrac{1}{n^t}\right)$$
	where $t>0$ and $N$ is a positive integer.
\end{lemma}
The proof of Lemma \ref{lem:xn} is supplied in Appendix \ref{app:B}.

\begin{theorem}[Convergence]\label{thm:convergence}
   Under Assumptions \ref{ass:A}-\ref{ass:uk}, the parameter estimate $\hat{\theta}_k$ given in algorithm \eqref{algorithm} converges to the true parameter of system \eqref{sys:theta} in mean square sense, i. e.
   $$\lim _{k \rightarrow \infty} E\{ e_{k}^{T} e_{k}\}=0,$$
   if 
   \begin{equation}\label{condition}
f_{m}\delta-\dfrac{ f_{M}\sqrt{g(h)} M_{2}M}{1-h_1}(M_{2}+2M_{1})>0,
\end{equation}
where $h$ is given in Assumption \ref{ass:A}, $g(h)$ is given in Remark \ref{rem:a2leqh}, $M$ and $h_1$ are given in Remark \ref{rem:Akfanshu},  $f_M$ and $f_m$ are the maximum and minimum of the function $f(x)$ on the interval $[C-BM_2, C+BM_2]\cup[C-BM_1, C+BM_1]$, $C$ is the quantization threshold, $M_1$ and $M_2$ are given in Remark \ref{rem:M12}, and $\delta$ is given in Lemma \ref{lem:delta}, respectively.
  \end{theorem}

\begin{proof}
By algorithm \eqref{algorithm}, we have
\begin{equation}\label{prooffirst}
	\begin{aligned}
		\left\|e_{k}\right\|  \leq&\left\|e_{k-1}+\dfrac{\gamma}{k} \varphi_{k}\left(F\left(C-\varphi_{k}^{T} \hat{\theta}_{k-1}\right)-s_{k}\right)\right\|. \\
	\end{aligned}
\end{equation}
About $s_{k}$, we have
$Es_{k} = {EI}_{\left\{ z_{k} \leq C \right\}} = {E(I}_{\left\{ d_{k} \leq C - \phi_{k}^{T}\theta \right\}} = F(C - \phi_{k}^{T}\theta)$.
There exists $ \xi_{k} $ between $ C - \varphi_{k}^{T}{\widehat{\theta}}_{k - 1}$ and $C - \phi_{k}^{T}\theta $ , which makes the following equation true
\begin{align}
&F\left( C - \varphi_{k}^{T}{\hat{\theta}}_{k - 1} \right) - F\left( C - \phi_{k}^{T}\theta \right)\notag\\
 &=-f\left( \xi_{k} \right) \cdot \left( (\varphi_{k}-\phi_{k})^{T}{\hat{\theta}}_{k -1}+\phi_{k}^{T}e_{k-1}  \right).
 \label{eq:central}
\end{align}
Because of $ C - BM_{2} \leq C - \varphi_{k}^{T}{\hat{\theta}}_{k - 1} \leq C + BM_{2} $ and $ C + BM_{1} \leq C - \phi_{k}^{T}\theta \leq C + BM_{1} $, $ f\left( x \right) $ has maximum and minimum values in the bounded area $ \left\lbrack C - BM_{2},C + BM_{2} \right\rbrack \cup \left\lbrack C + BM_{1},C + BM_{1} \right\rbrack $, recorded as $ {f_{M},f}_{m} $, hence $ f_{m}\leq f(x) \leq f_{M} $.

Let $R_{k}=E \{e_{k}^{T} e_{k}\} $, $V_{k}=E\left\lbrace\left(\varphi_{k}-\phi_{k}\right)^{T}\left(\varphi_{k}-\phi_{k}\right)\right\rbrace $, $ \hat{F}_{k}=F\left(C-\varphi_{k}^{T} \hat{\theta}_{k-1}\right)$. By \eqref{prooffirst} and \eqref{eq:central}, we have
\begin{align}\label{rk1}
& R_{k} \leq R_{k-1}\notag\\
  +&\dfrac{2\gamma}{k}E\left\lbrack \varphi_{k}^{T}e_{k - 1}\left( {\hat{F}}_{k} - s_{k} \right)  \right\rbrack+ \dfrac{\gamma^{2}}{k^{2}}E\left\lbrack \varphi_{k}^{T}\varphi_{k}\left( {\hat{F}}_{k} - s_{k} \right)^{2} \right\rbrack \notag\\
=& R_{k-1}+\dfrac{2\gamma}{k}E\left\lbrack \varphi_{k}^{T}e_{k - 1}\left( {\hat{F}}_{k} - s_{k} \right)  \right\rbrack+O\left( \dfrac{1}{k^{2}} \right) \notag\\
=& R_{k-1}-\dfrac{2\gamma f\left( \xi_{k} \right)}{k}E\left\lbrack (\varphi_{k}-\phi_{k}+\phi_{k})^{T}e_{k - 1}\times \right.\notag\\
&\left.\left((\varphi_{k}-\phi_{k})^{T}{\hat{\theta}}_{k -1}+\phi_{k}^{T}e_{k-1} \right)  \right\rbrack+O\left( \dfrac{1}{k^{2}} \right) \notag\\
=& R_{k-1}-\dfrac{2\gamma f\left( \xi_{k} \right)}{k}E \left\lbrack (\phi_{k}^{T}e_{k-1})^{2}+ \varphi_{k}^{T} e_{k - 1}(\varphi_{k}-\phi_{k})^{T}{\hat{\theta}}_{k -1}\right. \notag\\
& \left.+(\varphi_{k}-\phi_{k})^{T}e_{k - 1}\phi_{k}^{T}e_{k-1} \right\rbrack + O\left( \dfrac{1}{k^{2}}\right)
\notag\\
\leq & R_{k-1}-\dfrac{2\gamma f_{m}}{k}E(\phi_{k}^{T}e_{k-1})^{2}+\notag \\
&\dfrac{2\gamma f_{M}}{k}E \left\lbrack| \varphi_{k}^{T} e_{k - 1}(\varphi_{k}-\phi_{k})^{T}{\hat{\theta}}_{k -1}|+\right. \notag\\
&\left.|(\varphi_{k}-\phi_{k})^{T}e_{k - 1}\phi_{k}^{T}e_{k-1}|\right\rbrack+O\left( \dfrac{1}{k^{2}}\right).
\end{align}
Let $ \tilde{y}_{k}=\hat{y}_{k}-y_{k} $, by Remark \ref{rem:a2leqh} and Assumption \ref{ass:Theta}, we have 
$$
\begin{aligned}
&E\left|\left(\varphi_{k}-\phi_{k}\right)^{T} \hat{\theta}_{k-1}\right| \\
=&E\left|(\tilde{y}_{k-1}, \ldots, \tilde{y}_{k-p})(\hat{a}_{1}^{k-1}, \ldots, \hat{a}_{p}^{k-1})^{T}\right|\\
\leq &\sqrt{g(h) V_{k}} .
\end{aligned} 
$$
Similarly, $ E\left|\left(\left(\varphi_{k}-\phi_{k}\right)^{T} e_{k-1}\right)\right|\leq 2\sqrt{g(h) V_{k}} $. 
\begin{align}
R_{k}
\leq & R_{k-1}-\dfrac{2\gamma f_{m}}{k}E(\phi_{k}^{T}e_{k-1})^{2}+\dfrac{2\gamma f_{M}}{k} \times \notag \\
& \left(M_{2}\sqrt{g(h)R_{k-1}V_{k}} +2M_{1}\sqrt{g(h)R_{k-1}V_{k}}\right)+O\left( \dfrac{1}{k^{2}}\right)\notag \\
\leq & R_{k-1}-\dfrac{2\gamma f_{m}}{k}E(\phi_{k}^{T}e_{k-1})^{2}+\dfrac{2\gamma f_{M}}{k} \times \notag \\
& (M_{2}+2M_{1})\sqrt{g(h)}\sqrt{\alpha_{1}R_{k-1}\dfrac{1}{\alpha_{1}}V_{k}} +O\left( \dfrac{1}{k^{2}}\right)\notag \\
\leq & R_{k-1}-\dfrac{2\gamma f_{m}}{k}E(\phi_{k}^{T}e_{k-1})^{2}+\dfrac{\gamma f_{M}}{k} \times \notag\\
& (M_{2}+2M_{1})\sqrt{g(h)}\left(\alpha_{1}R_{k-1}+\dfrac{V_{k}}{\alpha_{1}}\right) +O\left( \dfrac{1}{k^{2}}\right).\notag \\
\end{align}
 Take the average of $ k,k + 1,\ldots,k + N - 1 $ moments, and let $ \overline{R}_{k} = \dfrac{R_{k}  + \cdots + R_{k+ N - 1}}{N}\),\(\overline{V}_{k}  = \dfrac{V_{k}  + \cdots + V_{k+N-1} }{N} $, we have
$$
\begin{aligned}
\overline{R}_{k}  \leq \overline{R}_{k-1} & - \dfrac{2\gamma f_{m}}{N\left( k + N - 1 \right)}\sum_{i = k}^{k + N - 1}{E\left( \phi_{i}^{T}e_{i - 1} \right)^{2}} \\
+\dfrac{\gamma f_{M}}{k}&(M_{2}+2M_{1})\sqrt{g(h)}
 \left( \alpha_{1}\overline{R}_{k-1} +\dfrac{\overline{V}_{k}}{\alpha_{1}} \right)+O\left( \dfrac{1}{k^{2}}\right). \\
\end{aligned}
$$
For any $ k \leq i,j \leq k + N - 1 $, by Lemma \ref{lem:theta}, we have
$$
\begin{aligned}
\sum_{i = k}^{k + N - 1}\left( \phi_{i}^{T}e_{i - 1} \right)^{2}
 = & \sum_{i = k}^{k + N - 1}{e_{j - 1}^{T}\phi_{i}\phi_{i}^{T}e_{j - 1} + O\left( \dfrac{1}{k} \right)} \\
 = & e_{j - 1}^{T}\left( \sum_{i = k}^{k + N - 1}{\phi_{i}\phi_{i}^{T}} \right)e_{j - 1} + O\left( \frac{1}{k} \right). \\
\end{aligned}
$$
By Lemma \ref{lem:delta}, we have
\begin{align}
 & - \dfrac{2\gamma f_{m}}{N\left( k + N - 1 \right)}\sum_{i = k}^{k + N - 1}\left( \phi_{i}^{T}e_{i - 1} \right)^{2} \notag\\
 = & - \dfrac{2\gamma f_{m}}{N^{2} k }N\sum_{i = k}^{k + N - 1}\left( \phi_{i}^{T}e_{i - 1} \right)^{2}+O\left( \dfrac{1}{k^{2}} \right) \notag\\
 = & - \dfrac{2\gamma f_{m}}{N^{2}k }\sum_{j = k}^{k + N - 1}{e_{j - 1}^{T}\left( \sum_{i = k}^{k + N - 1}{\phi_{i}\phi_{i}^{T}} \right)e_{j - 1}} + O\left( \dfrac{1}{k^{2}} \right)  \notag\\
 \leq & - \dfrac{2\gamma f_{m}\delta}{Nk}\sum_{j = k}^{k + N - 1}{e_{j - 1}^{T}e_{j - 1}} + O\left( \dfrac{1}{k^{2}} \right) \notag\\
 =& - \dfrac{2\gamma f_{m}\delta}{k}\overline{R}_{k-1}  + O\left( \dfrac{1}{k^{2}} \right).
\end{align}
Hence
\begin{equation}\label{rk2}
\begin{aligned}
\overline{R}_{k}  \leq & \left(1 - \dfrac{\gamma}{k}\left(2f_{m}\delta- f_{M}\sqrt{g(h)}\alpha_{1}\left(M_{2}+2M_{1}\right) \right)\right) \overline{R}_{k-1} \\
 &+\dfrac{\gamma f_{M}\sqrt{g(h)}}{ k\alpha_{1} }\left(M_{2}+2M_{1}\right)\overline{V}_{k}  + O\left( \dfrac{1}{k^{2}} \right). \\
\end{aligned}
\end{equation}
Then, we need to analyze the convergence of $\overline{V}_{k}$.
We already know
$$
\begin{aligned}
\tilde{y}_{k}
&=A(z)(\hat{y}_{k-1}-y_{k-1})+\left(\hat{A}_{k}(z)-A(z)\right)\hat{y}_{k-1}\\
 &\quad +\left(\hat{B}_{k}(z)-B(z)\right)u_{k}
=A(z)\tilde{y}_{k-1}+\varphi_{k}^{T}e_{k}.\\
\end{aligned}
$$
Let $ \beta_{k+1}=	\begin{pmatrix}
		\tilde{y}_{k},
		\tilde{y}_{k-1},
		\ldots,
		\tilde{y}_{k-p+1}
	\end{pmatrix}^T
 $. Then
$$
\begin{aligned}
\beta_{k+1}&=A\beta_{k}+(\varphi_{k}^{T}e_{k}, 0, \ldots, 0)^T\\
&\ \ \vdots\\
&=A^{k+1} \beta_{0}+\sum_{i=0}^{k} A^{i}(\varphi_{k-i}^{T}e_{k-i}, 0, \ldots, 0)^T.
\end{aligned}
$$
By Lemma \ref{fanshu} and Remark \ref{rem:Akfanshu}, we have $  \|A^k\|\leq M h_1^k $, hence
\begin{equation}
\begin{aligned}
\|\beta_{k+1}\|&\leq Mh_1^{k+1}\| \beta_{0}\|+\sum_{i=0}^{k} Mh_1^{i}|\varphi_{k-i}^{T}e_{k-i}|\\
&\leq Mh_1^{k+1}\| \beta_{0}\|+M_2 M\sum_{i=0}^{k} h_1^{i}\|e_{k-i}\|,\\
\end{aligned}
\end{equation}
and
\begin{equation}
\begin{aligned}
&V_{k+1}=E\{\beta_{k+1}^{T}\beta_{k+1}\}  \\
\leq& M^2 h_1^{2k+2}E\|\beta_0\|^2+2M^2 h_1^{k+1}M_2 E\left\lbrace \| \beta_{0}\| \sum_{i=0}^{k} h_1^{i}\|e_{k-i}\|\right\rbrace\\
&+M_2^2 M^2 \sum_{i=0}^{k}\sum_{j=0}^{k} h_1^{i+j}E\|e_{k-i}\|\|e_{k-j}\| \\
\leq& M^2 h_1^{2k+2}E\|\beta_0\|^2+2M^2 h_1^{k+1}M_2 E\left\lbrace \| \beta_{0}\| \sum_{i=0}^{k} h_1^{i}|e_{k-i}|\right\rbrace\\
&+M_2^2 M^2 \sum_{i=0}^{k}\sum_{j=0}^{k} \dfrac{1}{2}h_1^{i+j}(R_{k-i}+R_{k-j}). \\
\end{aligned}
\end{equation}
Since $|e_{k-i}|$ is bounded and $h_1<1$, $E\left\lbrace \| \beta_{0}\| \sum_{i=0}^{k} h_1^{i}|e_{k-i}|\right\rbrace$ is also bounded, then we can get
\begin{equation}\label{vk}
\begin{aligned}
&V_{k+1}\\
\leq& M_2^2 M^2 \sum_{i=0}^{k}\sum_{j=0}^{k} \dfrac{1}{2}h_1^{i+j}(R_{k-i}+R_{k-j}) +O\left( h_1^{k+1}\right)\\
=&  M_2^2 M^2  \dfrac{1-h_1^{k+1}}{1-h_1}\sum_{i=0}^{k}h_1^{i}R_{k-i}+O\left( h_1^{k+1}\right).
\end{aligned}
\end{equation}
Taking the average of N moments for \eqref{vk} and substitute it into \eqref{rk2} , we can get
\begin{equation}
\begin{aligned}
&\overline{R}_{k} \\
 \leq & \left(1 - \dfrac{\gamma}{k}\left(2f_{m}\delta- f_{M}\sqrt{g(h)}\alpha_{1}\left(M_{2}+2M_{1}\right) \right)\right) \overline{R}_{k-1} \\
 &+\! \dfrac{\gamma f_{M}\sqrt{g(h)}}{ k\alpha_{1}\left(1\! -\! h_1\right) }\left(M_{2}\! + \! 2M_{1}\right)M_2^2 M^2 \sum_{i=1}^{k}h_1^{i-1}\overline{R}_{k-i} \\
 &+ O\left( \dfrac{1}{k^{2}} \right). \\
\end{aligned}
\end{equation}
Let 
$ \eta_1\!=\!2\gamma f_m \delta -\gamma  f_{M}\sqrt{g(h)}\left(M_{2}\!+\!2M_{1}\right)\left(  \alpha_{1}\!+\!\dfrac{ M_2^2 M^2}{ \alpha_{1}\left(1\! -\! h_1\right) }  \right)  $, $ \eta_2=\dfrac{\gamma f_{M}\sqrt{g(h)}}{ \alpha_{1}\left(1\! -\! h_1\right) }\left(M_{2}\! + \! 2M_{1}\right)M_2^2 M^2  $
and reorganize the inequality above, then we have
\begin{equation}
\begin{aligned}
\overline{R}_{k}  \leq  \left(1 - \dfrac{\eta_{1}}{k}\right) \overline{R}_{k-1} 
 +\dfrac{\eta_{2}}{k} \sum_{i=2}^{k}h_1^{i-1}\overline{R}_{k-i} 
 + O\left( \dfrac{1}{k^{2}} \right). 
\end{aligned}
\end{equation}

By Lemmas \ref{lem:difference_rk} and \ref{lem:xn}, what must be satisfied is

 $ \eta=2\gamma f_m \delta -\gamma  f_{M}\sqrt{g(h)}\left(M_{2}\!+\!2M_{1}\right)\left(  \alpha_{1}\!+\!\dfrac{ M_2^2 M^2}{ \alpha_{1}\left(1\! -\! h_1\right) }  \right)-\dfrac{\gamma f_{M}\sqrt{g(h)}h_1}{ \alpha_{1}\left(1\! -\! h_1\right)^2 }\left(M_{2}\! + \! 2M_{1}\right)M_2^2 M^2>0.$

Choose $ \alpha_{1}=\dfrac{M_{2}M}{1-h_1} $ to maximize $ \eta $. Then the condition changes to
\begin{equation}\label{eta}
\eta=2\gamma\left(f_{m}\delta-\dfrac{ f_{M}\sqrt{g(h)} M_{2}M}{1-h_1}(M_{2}+2M_{1})\right)>0,
\end{equation}
which implies (\ref{condition}).
\end{proof}

\begin{remark}\label{rem:cond}
The condition \eqref{condition} in Theorem \ref{thm:convergence} is sufficient rather than necessary. Generally speaking, the condition is naturally satisfied for a relatively small $ h $.
\end{remark}

\begin{theorem}[Convergence rate]\label{thm:cr}
Under Assumptions \ref{ass:A}-\ref{ass:uk},
the mean square convergence rate of the estimate given in algorithm \eqref{algorithm} can achieve $O(1/k)$, i. e.
$$ E \{e_{k}^{T} e_{k}\}=O\left(\dfrac{1}{k}\right),$$
by choosing the step size to satisfy
\begin{equation}
\begin{aligned}
\gamma>\dfrac{1}{2\left(f_{m}\delta-\dfrac{ f_{M}\sqrt{g(h)}M_{2}M}{1-h_1}(M_{2}+2M_{1})\right)},
\end{aligned}
\end{equation}
 where $h$, $g(h)$, $\delta$, $M$, $h_1$, $f_m$, $f_M$, $M_1$ and $M_2$ are all the same as those in Theorem \ref{thm:convergence}.
\end{theorem}

\begin{proof}
By \eqref{eta} and Lemma \ref{lem:difference_rk},
let	
\begin{equation}\label{convrate}
\eta=2\gamma\left(f_{m}\delta-\dfrac{ f_{M}\sqrt{g(h)} M_{2}M}{1-h_1}(M_{2}+2M_{1})\right)>1,
\end{equation}
 then,
 $ \gamma>\dfrac{1}{2\left(f_{m}\delta-\dfrac{ f_{M}\sqrt{g(h)}M_{2}M}{1-h_1}(M_{2}+2M_{1})\right)}.$
\end{proof}

\begin{remark}\label{rem:var}
According to the definition of $ f_{M}$ and $f_m $, the variance of the noise cannot be too large or too small, otherwise it will affect convergence and convergence rate.
\end{remark}
\begin{remark}
We can find that when $ h$ converges to 0, the model (\ref{sys:ARMA1}) degenerates into FIR model, and the condition in Theorem \ref{thm:cr} changes to $ \gamma>\dfrac{1}{2f_{m}\delta}$, which is the same condition as that in Theorem 1 of \cite{wang2019adaptive}.
\end{remark}

\section{Simulation}

Consider the system with binary-valued observations
\begin{equation}\label{simsys}
\left\lbrace
\begin{matrix}
y_{k}=ay_{k-1}+bu_{k},\\
s_{k}=I_{\lbrace y_{k}+d_{k}\leq C \rbrace},
\end{matrix}\right.
\end{equation}
where $ \theta=[a,b]^{T} $ is unknown, the system noise $d_k \sim N(0,2) $, and the threshold $ C=0 $.

\subsection{The case with $ [a,b]^{T}=[-0.02,1]^{T} $}

\noindent(1) Convergence of the identification algorithm

Let $ \Omega=\lbrace (x_{1},x_{2})\big| | x_{1}|<0.03,0<x_{2}<1.1 \rbrace $, and $ \hat{\theta}_{0}=[0,0.9]^{T} $. The inputs is given as follows
 $u_k=1+0.01\omega_{k}$ as $k$ is odd and $u_k=0$ as $k$ is even,
where $\omega_{k}$ is uniformly distributed in $[0,1]$. By \eqref{simsys}, we can get that the outputs $y_k$ satisfies
$1\leq y_k \leq 1.01b/(1-a^2)$ as $k$ is odd and $|y_k|\leq (1.01\times |ab|)/(1-a^2)$ as $k$ is even.
By setting $ N=2$, we can get $\delta>0.47$, $h=0.03$,  $|\hat{y}_{k}|<(1.1\times 1.01)/(1-h^2)$, $ M_{1}^{2}<M_{2}^{2}<2.26$,
$f_{M}/f_{m}=e^{(B^{2}M_{2}^{2})/(2\sigma^{2})}<2 $.

Then, the condition of Theorem \ref{thm:convergence} is satisfied.
By the identification algorithm \eqref{algorithm} with $\gamma=3$, we can get the estimates of the parameters. Fig. \ref{FIG:00} shows that the estimate $ \hat{\theta}_{k} $ converges to the true parameter $ [-0.02,1]^{T} $, which is consistent with Theorem 1.

\begin{figure}[!h]
	\centering
\includegraphics[scale=0.57]{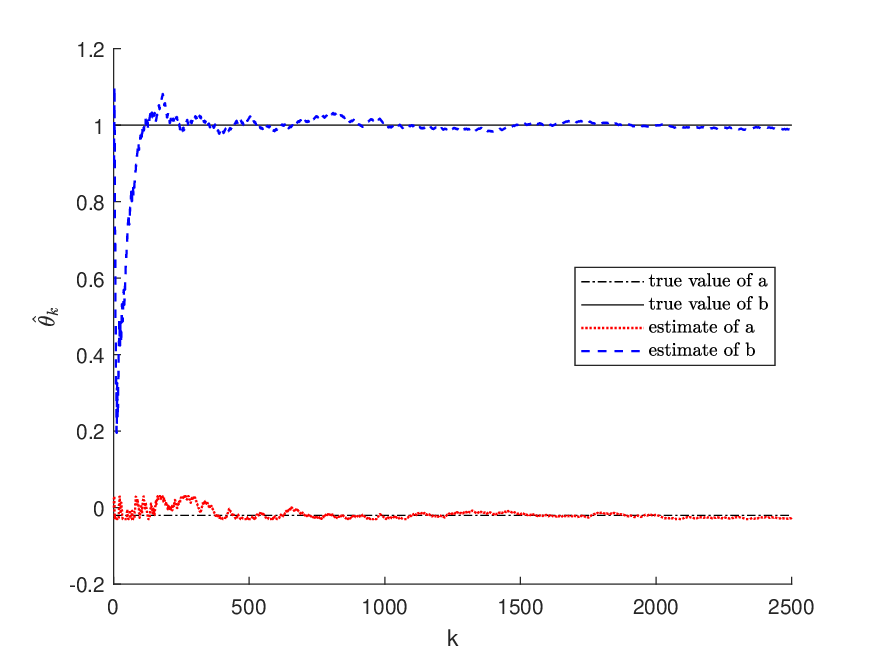}
	\caption{Trajectories of estimates}
	\label{FIG:00}
\end{figure}

Let $ \hat{\theta}_{0}=[0,0.7]^T $. The trajectories of the square of estimation error $\| e_k \|^2$ with different variances of noise are given in Fig. \ref{FIG:01}, which shows that too small noise variance will affect the convergence, and too large noise variance will affect the convergence rate ($ f_{m},f_{M} $ are affected).


\noindent(2) Convergence rate of the estimation error

 Fig. \ref{FIG:02} shows the logarithm of the parameter estimation error vs the logarithm of the index $k$ with different initial estimates.  From Fig. \ref{FIG:02}, we can see that the logarithms of the parameter estimation error with different initial estimates are all bounded by two linear functions of $\log (k)$, which implies that the convergence rate of $ e_{k}^{T} e_{k} $ is  $ O(1/k) $ and the initial estimates  has almost no effect on the convergence rate. These results are consistent with Theorem \ref{thm:cr}.

\begin{figure}[!h]
	\centering
\includegraphics[scale=0.57]{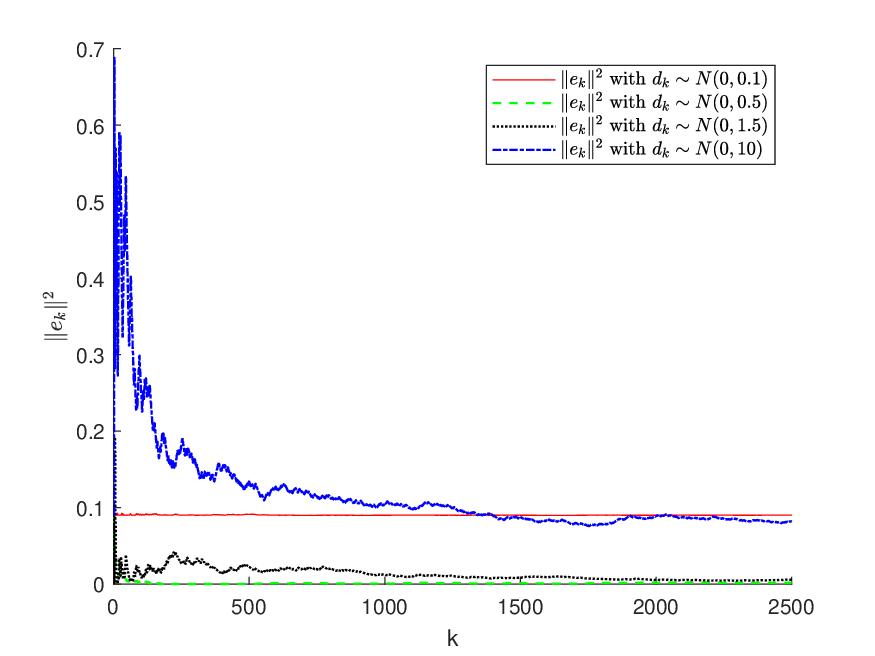}
	\caption{Trajectories of $\| e_k \|^2$ with different variances of noise}
	\label{FIG:01}
\end{figure}

\begin{figure}[!h]
	\centering
		\includegraphics[scale=0.55]{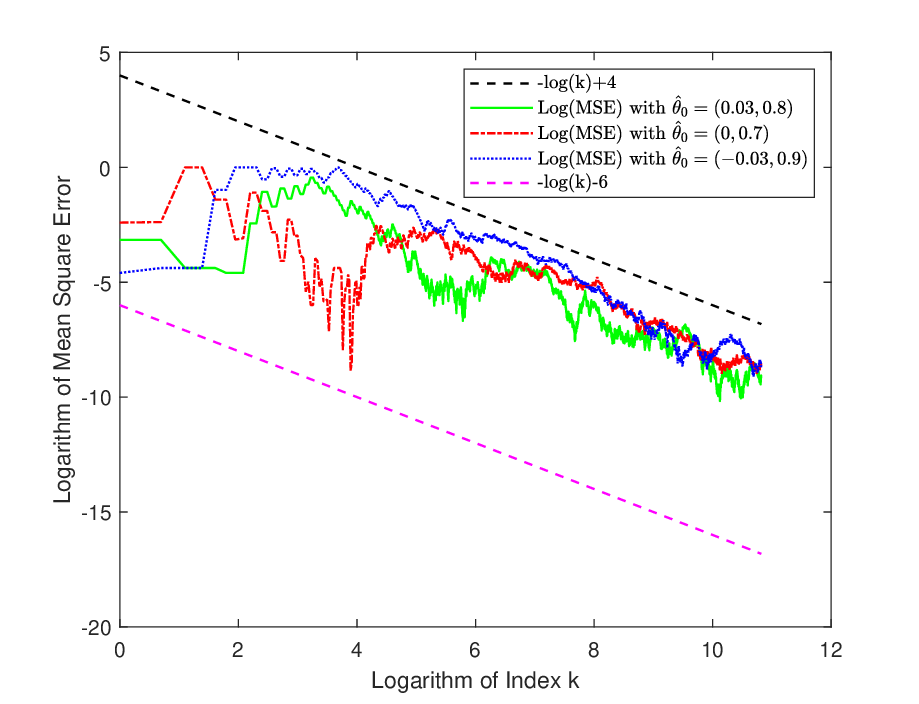}
	\caption{Trajectories of Log(MSE) with different $ \hat{\theta}_{0} $}
	\label{FIG:02}
\end{figure}

\begin{figure}[!h]
\centering
\includegraphics[scale=0.57]{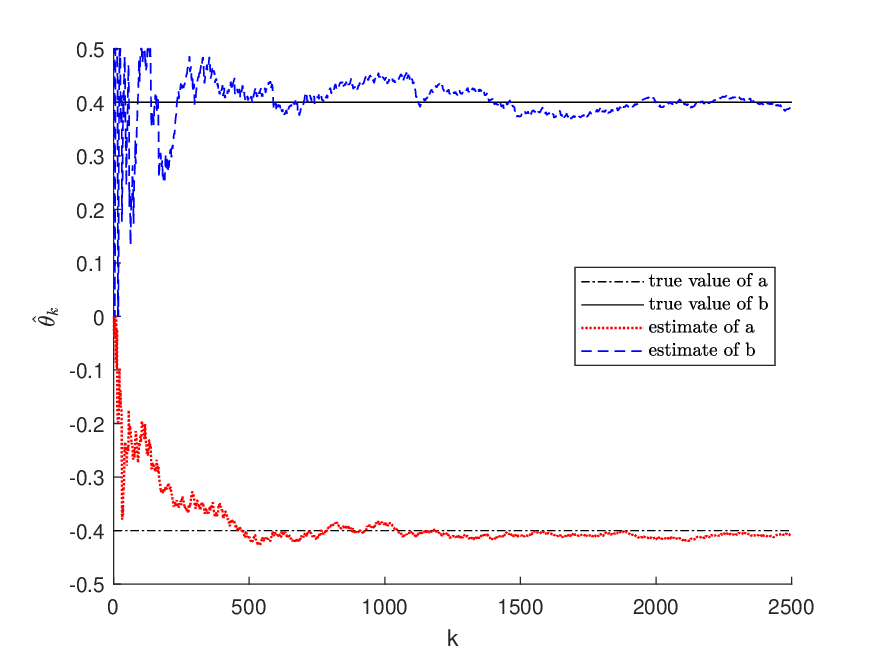}
    \caption{Trajectories of estimates with $u=5$}
	\label{FIG:2}
\end{figure}

\subsection{ The case with $[a,b]^{T}=[-0.4,0.4]^{T} $}

The conclusion of our theorem is sufficient but not necessary. Consider another system with parameters $\theta=[a,b]^{T}=[-0.4,0.4]^{T} $ and the inputs
$u_k=5$ as $k$ is odd and $u_k=0$ as $k$ is even. Let $\gamma=2$,  $ \hat{\theta}_{0}=[-0.5,0.25]^{T} $. Obviously, the condition is not satisfied at this time. However, Fig. \ref{FIG:2} shows the trajectories of the estimates $\hat{\theta}_{k}$ that can achieve to the true parameter.

\section{Conclusions}
This paper further studies the identification of high-dimensional ARMA models with binary-valued observations based on \cite{wang2023identification}. An online identification algorithm is proposed, which consists of parameter estimation and prediction of the original output. Compared to \cite{wang2023identification}, a more accurate upper bound for the output estimation error is given and based on this, it is proved that the estimates can converge to the true parameter in mean square sense. The convergence rate of the proposed identification algorithm for high-dimensional ARMA models can achieve as the same convergence rate O(1/k) as the algorithm in \cite{wang2019adaptive} for FIR models.

There are many meaningful future works. If the variance of noise is unknown, it can be estimated along with the parameters (see \cite{le2006noise}) by the proposed algorithm in this paper. It is worth studying whether the estimates still have good properties. Besides, the identification of ARMA models with multi-threshold quantized observations and the relationship between convergence rate and quantization level can be studied by referring to \cite{guo2014id}.  What's more, there are some other interesting problems. For examples, how to construct an optimal identification algorithm in the sense of CR lower bound? How to apply the proposed algorithm to adaptive tracking control of systems with binary-valued observations?

\section*{Appendix}
\appendix

\section{Proof of Lemma \ref{lem:difference_rk}}\label{app:A}
\begin{proof}
From \eqref{rkequation} and
\begin{equation}\label{rk-1equation}
r_{k-1}  =  \left(1 - \dfrac{\eta_{1}}{k-1}\right) r_{k-2} 
 +\dfrac{\eta_{2}}{k-1} \sum_{i=3}^{k}h_1^{i-2}r_{k-i} 
 + O\left( \dfrac{1}{k^{2}} \right),
\end{equation}
we get 
\eqref{rkequation}$-\dfrac{(k-1)h_1}{k} \times$  \eqref{rk-1equation} is
\begin{align*}
r_{k}  =&  \left(1 +h_1 - \dfrac{h_{1}+ \eta_{1}}{k}\right) r_{k-1} \\
& +h_1 \left( -1 +\dfrac{\eta_{1}+\eta_{2}+1}{k}\right) r_{k-2} 
 + O\left( \dfrac{1}{k^{2}} \right).
\end{align*}
By solving the characteristic equation, we can obtain
\begin{align*}
&r_{k}-\left(1-\dfrac{1}{k}\left(\eta_1 -\dfrac{\eta_{2}h_{1}}{1-h_{1}}\right)\right)r_{k-1}  \\
=& \left(h_1 \!-\!\dfrac{h_1\!-\!h_1^2\!+\!\eta_2 h_1}{k(1-h_1)}\right) \left[ r_{k-1}\!-\!\Bigg(1\!-\!\dfrac{\eta_1 \!-\!\dfrac{\eta_{2}h_{1}}{1\!-\!h_{1}}}{k\!-\!1}\Bigg)r_{k-2}\right] \\
&  + O\left( \dfrac{1}{k^{2}} \right).
\end{align*}
Let $z_k=r_{k}-\left(1-\dfrac{1}{k}\left(\eta_1 -\dfrac{\eta_{2}h_{1}}{1-h_{1}}\right)\right)r_{k-1}$, then
\begin{align*}
z_k=&\left(h_1 \!+O\left( \dfrac{1}{k} \right)\right) z_{k-1}+O\left( \dfrac{1}{k^2} \right)  \\
=&\left(h_1 \!+O\left( \dfrac{1}{k} \right)\right)\left(h_1 \!+O\left( \dfrac{1}{k-1} \right)\right) z_{k-2}\\
&+\left(h_1 \!+O\left( \dfrac{1}{k} \right)\right)O\left( \dfrac{1}{(k-1)^2} \right)+ O\left( \dfrac{1}{k^2} \right)  \\
=&O\left( \sum_{i=1}^{k}\dfrac{h_1^{k-i}}{i^2} \right)  \\
=& O\Bigg( \dfrac{ \sum_{i=1}^{k}\dfrac{h_1^{-i}}{i^2}  }{h_1^{-k}} \Bigg) \\
=& O\Bigg( \dfrac{ \dfrac{h_1^{-k}}{k^2}  }{h_1^{-k}-h_1^{-k+1}} \Bigg) \\
=& O\left(  \dfrac{1}{k^2}  \right).
\end{align*}
Thus, $r_{k}=\left(1-\dfrac{1}{k}\left(\eta_1 -\dfrac{\eta_{2}h_{1}}{1-h_{1}}\right)\right)r_{k-1} + O\left(  \dfrac{1}{k^2}  \right)$.
For the iteration $r_{k}=\left(1-\dfrac{\eta}{k}\right)r_{k-1} + O\left(  \dfrac{1}{k^2}  \right)$, $ r_{k} $ converges if $ \eta>0 $, and we have
\begin{align*}
r_{k}=&\left(1-\dfrac{\eta}{k}\right)r_{k-1}+O\left(\dfrac{1}{k^2}\right) \\
=& \prod_{l=1}^{k}\left(1-\dfrac{\eta}{l}\right)r_{0}+\sum_{l=1}^{k}\prod_{i=l}^{k}\left(1-\dfrac{\eta}{i}\right)O\left(\dfrac{1}{l^{2}}\right)\\
=&O\left(\dfrac{1}{k^{\eta}}\right) + \sum\limits_{l=1}^{k}O\left( \Big(\dfrac{l+1}{k}\Big)^{\eta}\right) O\left(\dfrac{1}{l^{2}}\right)\\
=&\begin{cases}
O\left(\dfrac{1}{k^{\eta}}\right), &0<\eta<1;\\
O\left(\dfrac{ln k}{k}\right), &\eta=1;\\
O\left(\dfrac{1}{k}\right), &\eta>1.
\end{cases}
\end{align*}
\end{proof}

\section{Proof of Lemma \ref{lem:xn}}\label{app:B}
\begin{proof}
i) If $\dfrac{x_n+\cdots+x_{n+N-1}}{N}=O\left(\dfrac{1}{n^t}\right),$
 then
$$ 0 \leq \dfrac{x_n}{N}\leq\dfrac{x_n+\cdots+x_{n+N-1}}{N}=O\left(\dfrac{1}{n^t}\right)$$
$ \Rightarrow x_n=O\left(\dfrac{1}{n^t}\right) .$

ii) If $x_n=O\left(\dfrac{1}{n^t}\right),$
 then
$ x_{n+l}=O\left(\dfrac{1}{(n+l)^t}\right)= O\left(\dfrac{1}{n^t}\right), l=1, \ldots, N-1.$
Hence
$\dfrac{x_n+\cdots+x_{n+N-1}}{N}=O\left(\dfrac{1}{n^t}\right).$
The lemma is proved.
\end{proof}

%

\begin{thebibliography}{22}
\providecommand{\natexlab}[1]{#1}
\providecommand{\url}[1]{\texttt{#1}}
\expandafter\ifx\csname urlstyle\endcsname\relax
  \providecommand{\doi}[1]{doi: #1}\else
  \providecommand{\doi}{doi: \begingroup \urlstyle{rm}\Url}\fi

\bibitem[Bi et~al.(2015)Bi, Kang, Zhao, et~al.]{bi2015svsi}
Wenjian Bi, Guolian Kang, Yanlong Zhao, et~al.
\newblock {SVSI: Fast and powerful set-valued system identification approach to
  identifying rare variants in sequencing studies for ordered categorical
  traits}.
\newblock \emph{Annals of human genetics}, 79\penalty0 (4):\penalty0 294--309,
  2015.

\bibitem[Bradfield et~al.(2012)Bradfield, Taal, Timpson,
  et~al.]{bradfield2012genome}
Jonathan~P. Bradfield, H.~Rob Taal, Nicholas~J. Timpson, et~al.
\newblock A genome-wide association meta-analysis identifies new childhood
  obesity loci.
\newblock \emph{Nature genetics}, 44\penalty0 (5):\penalty0 526, 2012.

\bibitem[Calamai and Mor{\'e}(1987)]{calamai1987projected}
Paul~H. Calamai and Jorge~J. Mor{\'e}.
\newblock Projected gradient methods for linearly constrained problems.
\newblock \emph{Mathematical programming}, 39\penalty0 (1):\penalty0 93--116,
  1987.

\bibitem[Casini et~al.(2007)Casini, Garulli, and Vicino]{casini2007time}
Marco Casini, Andrea Garulli, and Antonio Vicino.
\newblock Time complexity and input design in worst-case identification using
  binary sensors.
\newblock In \emph{46th IEEE Conference on Decision and Control}, pages
  5528--5533, 2007.

\bibitem[Chen and Guo(1987)]{chen1987adaptive}
Han-Fu Chen and Lei Guo.
\newblock Adaptive control via consistent estimation for deterministic systems.
\newblock \emph{International Journal of Control}, 45\penalty0 (6):\penalty0
  2183--2202, 1987.

\bibitem[Godoy et~al.(2011)Godoy, Goodwin, Ag{\"u}ero, Marelli, and
  Wigren]{godoy2011identification}
Boris~I. Godoy, Graham~C. Goodwin, Juan~C. Ag{\"u}ero, Dami{\'a}n Marelli, and
  Torbj{\"o}rn Wigren.
\newblock On identification of {FIR} systems having quantized output data.
\newblock \emph{Automatica}, 47\penalty0 (9):\penalty0 1905--1915, 2011.

\bibitem[Guo and Zhao(2013)]{guo2013recursive}
Jin Guo and Yanlong Zhao.
\newblock Recursive projection algorithm on {FIR} system identification with
  binary-valued observations.
\newblock \emph{Automatica}, 49\penalty0 (11):\penalty0 3396--3401, 2013.

\bibitem[Guo and Zhao(2014)]{guo2014id}
Jin Guo and Yanlong Zhao.
\newblock Identification of the gain system with quantized observations and
  bounded persistent excitations.
\newblock \emph{Science China Information Sciences}, 57\penalty0
  (012205):\penalty0 1--15, 2014.

\bibitem[Marelli et~al.(2013)Marelli, You, and Fu]{marelli2013identification}
Dami{\'a}n Marelli, Keyou You, and Minyue Fu.
\newblock Identification of {ARMA} models using intermittent and quantized
  output observations.
\newblock \emph{Automatica}, 49\penalty0 (2):\penalty0 360--369, 2013.

\bibitem[Wang et~al.(2002)Wang, Kim, and Sun]{wang2002prediction}
Le~Yi Wang, Yong-Wha Kim, and Jing Sun.
\newblock Prediction of oxygen storage capacity and stored {NOx by HEGO}
  sensors for improved {LNT} control strategies.
\newblock In \emph{ASME International Mechanical Engineering Congress and
  Exposition}, 777--785, 2002.

\bibitem[Wang et~al.(2003)Wang, Zhang, and Yin]{le2003system}
Le~Yi Wang, Ji-Feng Zhang, and G.~George Yin.
\newblock System identification using binary sensors.
\newblock \emph{IEEE transactions on automatic control}, 48\penalty0
  (11):\penalty0 1892--1907, 2003.

\bibitem[Wang et~al.(2006)Wang, Yin, and Zhang]{le2006noise}
Le~Yi Wang, G.~George Yin, and Ji-Feng Zhang.
\newblock Joint identification of plant rational models and noise distribution
  functions using binary-valued observations.
\newblock \emph{Automatica}, 42:\penalty0 535--547, 2006.

\bibitem[Wang et~al.({2010})Wang, Yin, Zhang, and Zhao]{le2010system}
Le~Yi Wang, G.~George Yin, Ji-Feng Zhang, and Yanlong Zhao.
\newblock \emph{System identification with quantized observations}.
\newblock Springer, {2010}.

\bibitem[Wang et~al.(2016)Wang, Bi, Zhao, and Xue]{wang2016radar}
Ting Wang, Wenjian Bi, Yanlong Zhao, and Wenchao Xue.
\newblock Radar target recognition algorithm based on {RCS} observation
  sequence -- set-valued identification method.
\newblock \emph{Journal of Systems Science and Complexity}, 29\penalty0
  (3):\penalty0 573--588, 2016.

\bibitem[Wang et~al.(2018)Wang, Tan, and Zhao]{wang2018asymptotically}
Ting Wang, Jianwei Tan, and Yanlong Zhao.
\newblock Asymptotically efficient non-truncated identification for {FIR}
  systems with binary-valued outputs.
\newblock \emph{Science China Information Sciences}, 61\penalty0 (12):\penalty0
  129208, 2018.

\bibitem[Wang et~al.(2021)Wang, Hu, and Zhao]{wang2019adaptive}
Ting Wang, Min Hu, and Yanlong Zhao.
\newblock {Adaptive tracking control of FIR systems under binary-valued
  observations and recursive projection identification}.
\newblock \emph{IEEE Transactions on Systems, Man, and Cybernetics: Systems},
  51\penalty0 (9):\penalty0 5289--5299, 2021.

\bibitem[Woodward et~al.(2017)Woodward, Gray, and Elliott]{woodward2017applied}
Wayne~A Woodward, Henry~L Gray, and Alan~C Elliott.
\newblock \emph{Applied time series analysis with {R}}.
\newblock CRC press, 2017.

\bibitem[You(2015)]{you2015recursive}
Keyou You.
\newblock Recursive algorithms for parameter estimation with adaptive
  quantizer.
\newblock \emph{Automatica}, 52:\penalty0 192--201, 2015.

\bibitem[Yu et~al.(2016)Yu, You, and Xie]{yu2016quantized}
Chengpu Yu, Keyou You, and Lihua Xie.
\newblock Quantized identification of {ARMA} systems with colored measurement
  noise.
\newblock \emph{Automatica}, 66:\penalty0 101--108, 2016.

\bibitem[Lei.(2020)]{Lei2020}
Lei Guo. \emph{Time-Varying Stochastic Systems: Stability and Adaptive Theory}. Beijing: Science Press, 2020.

\bibitem[Zhang et~al.(2019)Zhang, Wang, and Zhao]{zhang2019fir}
Hang Zhang, Ting Wang, and Yanlong Zhao.
\newblock {FIR} system identification with set-valued and precise observations
  from multiple sensors.
\newblock \emph{Science China Information Sciences}, 62\penalty0 (5):\penalty0
  52203, 2019.

\bibitem[Zhang et~al.(2021)Zhang, Wang, and Zhao]{zhang2019asymptotically}
Hang Zhang, Ting Wang, and Yanlong Zhao.
\newblock Asymptotically efficient recursive identification of {FIR} systems
  with binary-valued observations.
\newblock \emph{IEEE Transactions on Systems, Man, and Cybernetics: Systems},
  51\penalty0 (5):\penalty0 2687--2700, 2021.

\bibitem[Zhao et~al.(2016)Zhao, Bi, and Wang]{zhao2016iterative}
Yanlong Zhao, Wenjian Bi, and Ting Wang.
\newblock Iterative parameter estimate with batched binary-valued observations.
\newblock \emph{Science China Information Sciences}, 59\penalty0 (5):\penalty0
  052201, 2016.

\bibitem[Wang et~al.(2023)Wang, Li, Guo and Zhao]{wang2023identification}
Ting Wang, Xin Li, Jin Guo, and Yanlong Zhao. Identification of ARMA models with binary-valued observations. \emph{Automatica}, 149, 110832, 2023.

\bibitem[Brockwell, et~al.(2016)]{Brockwell2016ARMA}
Peter J. Brockwell, Richard A. Davis (2016). ARMA Models. In: Introduction to Time Series and Forecasting. Springer Texts in Statistics. Springer, Cham.
\end{thebibliography}

\end{document}